\documentclass[12pt]{article}

\usepackage{amsmath}
\usepackage{amssymb}
\usepackage{latexsym}
\usepackage{amsthm}
\usepackage{enumerate}
\usepackage{my}
\usepackage{pstricks}
\usepackage[dvips]{graphicx}

\numberwithin{equation}{section}
\theoremstyle{plain}
\newtheorem{theorem}{Theorem}[section]

\newtheorem{cor}[theorem]{Corollary}

{\theoremstyle{definition}

\newtheorem{remark}[theorem]{Remark}

}

\renewcommand{\P}[1]{{\mathbf{P}}^{#1}}

\title{Cubic Curves, Finite Geometry and Cryptography}

\author{A.A. Bruen, J.W.P. Hirschfeld, and D.L. Wehlau}
\date{}

\begin{document}
\maketitle

\begin{abstract}
Some geometry on non-singular cubic curves, mainly over finite fields, is
surveyed. Such a curve has $9,3,1$ or $0$ points of inflexion, and cubic
curves are classified accordingly. The group structure and the possible
numbers of rational points are also surveyed. A possible strengthening of the
security of elliptic curve cryptography is proposed using a `shared secret'
related to the group law. Cubic curves are also used  in a new way to construct
sets of points having various combinatorial and geometric properties that are
of particular interest in finite Desarguesian planes.
\end{abstract}


\section{Introduction}
\label{sec1}
In cryptography, the following views of an elliptic curve over a particular
field $K$ are common:
\begin{enumerate}
\item[(i)] a curve of genus 1;

\item[(ii)] a plane non-singular cubic curve;

\item[(iii)] a plane non-singular cubic curve with an inflexion;

\item[(iv)] $\{(x,y)\mid y^2 = x^3 + ax + b  \}$.
\end{enumerate}

In this paper  (iii) is used; for many fields, it is equivalent to (iv).

However, to perform elliptic curve cryptography (ECC) on a non-singular cubic
curve it is really not necessary to assume that the curve has an inflexion.
This then widens the choice of the curve that is used for the encryption.
%
%
If only non-singular cubics with more than one inflexion point are
considered then, since an inflexion point other than the zero has
order 3 and the order of a subgroup divides the order of the group,
this restricts to curves whose group size is divisible by 3.

Given two irreducible curves $\cC,\,\cD$, an {\em isomorphism} from $\cC$ to
$\cD$ is an invertible polynomial transformation; it induces an isomorphism of
their function fields. A non-singular cubic curve is isomorphic to one
containing at least one inflexion point; see, for example, \cite[Section 7.10]{HKT}. Two
non-singular cubics, both with at least one inflexion, are isomorphic if and
only if there is a projective transformation between them. So to classify
non-singular cubics up to isomorphism is equivalent to classifying
non-singular cubics with an inflexion up to projective transformation.

Given a non-singular cubic with an inflexion, when the field $K$ has
characteristic other than two, co-ordinates may be chosen so that the line at
infinity contains an inflexion point and the curve is normalised to the form
$y^2 =f(x)$, where $f$ has degree 3. Canonical forms for these cubics are
given in Section \ref{subsec6.3}. However, there do exist non-singular cubic
curves having no inflexion point; see Section \ref{subsec6.4}.

Also, in this paper, a modification of the usual version of elliptic curve
cryptography is suggested. Suppose two parties $A$ and $B$ are establishing a
secret key using elliptic curve cryptography. They are working with a given
cubic curve $\cC$; it may be singular or non-singular and it may or may not
have an inflexion. It may also be noted that elliptic curve cryptography may
be carried out over any finite field using any cubic curve. In the usual
version of elliptic curve cryptography, the line at infinity is a tangent at
the inflexion $O = (0 : 1 : 0)$. The identity element for the group structure
is always chosen to be the inflexion point $O$. In the proposed variation, $A$
and $B$ share a secret. This secret, which will be digitised, is the choice of
the identity element which is known only to $A$ and $B$ and which can be any
point of the curve $\cC$. The choice of this identity point determines the
group operation. The unknown identity of the identity point then makes the
task of an eavesdropper that much more difficult. 

  In Section~\ref{appl},   some new and purely geometrical
  applications of cubic curves over finite fields are discussed.

\section{Projective plane curves}
\label{sec2}

Let $K$ be any field and let $\ov{K}$ be its  the algebraic closure. 
Let $F(X,Y,Z)$ be a form, that is, a homogeneous polynomial in
$K[X,Y,Z]$. The graph of this form, 
$$\cC= \{(x:y:z) \in
\P2(K) \mid F(x,y,z)=0\},
$$ is a {\em curve} in the projective plane
$\P2(K)$. The curve is  {\em irreducible} if $F(X,Y,Z)$ does not
factor in $\ov{K}[X,Y,Z]$.

 A point $P$ lying on a curve is a {\em singular point} of the curve if there
 is more than one tangent line to the curve through $P$, \cite[Section
 1.3]{HKT}. If no such point exists in $\P2(\ov{K})$, that is, if there is a
 unique tangent line at each point of the curve considered over $\ov{K}$, then
 the curve is {\em non-singular}. This means that, working over the algebraic
 closure of $K$, it is impossible to find a point $P$ on $\cC$ such that the
 three partial derivatives of $F$ with respect to $X,Y,Z$ are all zero at $P$.
 If a curve $\cC$ in $\P2(K)$ has a singular point in $\P2(\ov{K})$ then the
 curve $\cC$ is {\em singular}.

Geometrically, the non-singularity of $\cC$ means that it has no node or cusp
or isolated double point; so there is a unique tangent line to the curve at
every point $P$. 

\section{Inflexion points}
\label{sec3}

 A {\em point of inflexion} $P$ of a curve is one for which the tangent at $P$
 has triple contact with the curve, \cite[Section 1.3]{HKT}. Thus, in
 particular, the tangent line at an inflexion $P$ of a cubic curve has no
 other point in common with the curve. 

 The condition that the tangent line at $P$ has triple contact with the curve
 is expressed algebraically by the requirement that 
 $$
 F(X,Y,Z) = f(X,Y,Z)\cdot
 g(X,Y,Z) + (aX+bY+cZ)^3 h(X,Y,Z),
 $$ 
 where 
 \begin{enumerate}
\item[(i)]
$f(X,Y,Z)$ is the linear form
 defining the tangent line at $P$,  
\item[(ii)] $g(X,Y,Z)$ is some form of degree
 $n-1$, 
 \item[(iii)] $h(X,Y,Z)$ is a form of degree $n-3$, 
 \item[(iv)] $aX+bY+cZ$ is some linear form
 vanishing at $P$,
 \item[(v)] $n$ is the degree of the form $F$. 
\end{enumerate} 

For cubic curves, Points of inflexion are considered in relation to the group
structure in Section~\ref{grouplaw}.

Over any field, the line joining any two inflexions meets the
cubic in a third inflexion. To see this result the following
Theorem of the Nine Associated Points is used.
\begin{theorem}
\label{nine} 
 Let $\cE$ be an irreducible cubic curve defined over $K$ by
  $E$ and suppose that $\cD$ and $\cD'$  are any two other cubic 
  curves defined over $K$ by the forms $D$ and $D'$.
 If 
 \begin{eqnarray*}
 \cE\cdot \cD & = & P_1 + P_2 + \cdots + P_9, \\
 \cE\cdot \cD' & = & P_1 + P_2 + \cdots + P_8 + R,
\end{eqnarray*} 
 then $R=P_9$.
\end{theorem}

\begin{proof} Here, the classical proof is given in the case that the $P_i$ 
are distinct. The general proof follows from Noether's Theorem; see 
Fulton \cite[Section 5.6]{F} or \cite[Section 4.5]{HKT}. 
  The general cubic form $F(X,Y,Z)$ has 10 coefficients. The 8 equations
  $E(P_i)=0$ for $i=1,2,\ldots,8$ impose 8 linearly independent conditions on
  the form $E$. So there is a pencil of cubics
  which pass through the 8 points $P_1,P_2,\dots,P_8$. Hence
  $D' = \ga E + \gb D$ for some $\ga,\gb \in K$.
  Since $E(P_9)= D(P_9)=0$, so $D'(P_9)=0$. Therefore $R = P_9$.
\end{proof}

\begin{theorem}
\label{flex3}
  Let $\cC$ be a cubic defined over $K$. If $P_1,P_2$ are two
  distinct inflexion points of $\cC$ lying in $\P2(K),$ and the line $\ell
  =P_1P_2$ meets $\cC$ again in $P_3,$ then $P_3$ is also an inflexion point of
  $\cC$.
\end{theorem}

\begin{proof}
  Let $\ell_i$ be the tangent line to $\cC$ at the point $P_i,\ i=1,2,3 $, 
  and let $\cC \cdot \ell_3 = 2P_3 + R$.  Define two cubics
  $\cD=\ell^3$ and   $\cD'=\ell_1\ell_2\ell_3$.
  Then 
  \begin{eqnarray*}
  \cC\cdot\cD & = & 3P_1 + 3P_2 + 3P_3,\\ 
  \cC\cdot\cD' & = & 3P_1 + 3P_2 + 2P_3 + R.
  \end{eqnarray*}  
  So, by the previous
  theorem, $R=P_3$; that is, $\cC \cdot \ell_3 = 3P_3$ and thus $P_3$ is
  an inflexion point of $\cC$.
\end{proof}

  Given a form $F(x,y,z)$ of degree $n$, its {\em Hessian $\cH$} is defined as
  the curve given by the form $H$ that is the  determinant of the 
  second-order partial derivatives of
  $F$: $$ H = \left|
  \begin{matrix}
    F_{XX} & F_{XY} & F_{XZ}\\
    F_{YX} & F_{YY} & F_{YZ}\\
    F_{ZX} & F_{ZY} & F_{ZZ}
  \end{matrix}
   \right| \ .
   $$  Thus the Hessian is  a curve of degree $3(n-2)$.

  \begin{theorem}
  \label{hess}
    Suppose $F(X,Y,Z)\in K[X,Y,Z]$ is a form of degree $n$ and that 
    $2(n-1)$ is invertible in $K$.
    A non-singular point $P$ lying on the curve $\cC$ defined by $F$ is an 
    inflexion point of $\cC$
    if and only if its Hessian form $H$ vanishes at $P$.
  \end{theorem}

  \begin{remark}
    If $2(n-1)$ is not invertible in $K$ then $H$ is identically zero. For
    an appropriate treatment in this case, see \cite[Section 11.2]{James}.
  \end{remark}

  Suppose now that $\cC$ is a non-singular cubic curve. Then its Hessian $\cH$
  is also a cubic. B{\'e}zout's Theorem shows that, over an algebraically
  closed field of characteristic different from $2$ and $3$, there are, in
  general, $9$ inflexion points of $\cC$. 
  
Over the field of complex numbers these nine points form a famous
configuration, namely the 9 points of $\mathrm{AG}(2,3)$, the affine plane of
order 3. Classically this $(9_4,12_3)$ configuration of 9 points and 12 lines,
with 4 lines through a point and 3 points on a line, is called the Hesse
Configuration. Over a finite field $K= \Fq$ there are $0, 1, 3$ or $9$
inflexions. In the case of 9 inflexions, the 9 points again form a copy of
$\AG(2,3)$ embedded in the projective plane $\PG(2,q)$.

\begin{theorem}
\label{flexno}
The number of rational inflexions of a non-singular cubic over $\Fq$ is
$0,1,3,$ or $9$. The possibilities are as follows$:$
\[
\begin{array}{ll}
q\equiv 0\ ({\rm mod}\ 3): & \quad 0,\ 1,\ 3;\\
q\equiv 2\ ({\rm mod}\ 3): & \quad 0,\ 1,\ 3;\\
q\equiv 1\ ({\rm mod}\ 3): & \quad 0,\ 1,\ 3,\ 9.
\end{array}
\]
\end{theorem}

In the case that $q\equiv 1\pmod 3$, by a suitable choice of coordinates, the
configuration of $9$ points always has the following form $\cK_9$, where $\go$
is a primitive cube root of unity in $K$:
\begin{eqnarray}
\cK_9 &=&\{(0,1,-1),(0,1,-\go), (0,1,-\go^2), (-1,0,1),(-\go,0,1),\nonumber\\
&& \qquad(-\go^2,0,1),(1,-1,0),(1,-\go,0),(1,-\go^2,0)\}\label{k9}
\end{eqnarray}
This set  $\cK_9$
is a Hessian  configuration for the non-singular cubic with form
\[
F = X^3+Y^3 +Z^2 -3c XYZ,
\]
with $c$ any element such
that $c^3 \neq 1$.
When $q=4$, take $c=0$; then $\cK_9$ is the set of rational points
of  the Hermitian  curve with form
 $$
 X \ov{X}+Y \ov{Y}+ Z\ov{Z},
 $$
where $\ov{T} = T^\sqq = T^2.$

For further illumination on inflexions of a cubic, including the singular ones,
see \cite[Chapter 11]{James}.

The advent of elliptic curve cryptography has aroused considerable interest in
elliptic curves over $\Fq$. The main idea involves a key-exchange between two
communicating parties, similar to the Diffie--Hellman protocol. There is a
publicly prescribed elliptic curve over some finite field, with associated
group $G$ that may be taken to be cyclic, with generator $P$. Communicating
parties $A,B$ choose their secret positive numbers $\ga,\gb$. Then $A$ openly
transmits the point $\ga P$, that is, $P$ added to itself $\ga$ times, to $B$.
Also, $B$ transmits $\gb P$ in the open to $A$. Now, $A$ calculates $\ga(\gb
P)$ and $B$ calculates $\gb(\ga P)$. The upshot is that $A$ and $B$ are now in
possession of a common secret key $\ga\gb P=\gb\ga P$. Security rests on the
unproved assumption that, given $mP$, it is not possible to calculate $m$ in a
`reasonable' amount of time. The commercialisation of this key-exchange has
led to an intensive study of elliptic curves over a finite field.

\section{The group law on a cubic}
\label{grouplaw}
Let $\cC$ be an irreducible cubic curve in $\P2(K)$, and consider only the
{\em rational} points of $\cC$, that is, those lying over $K$; denote this set
by $\cC(K)$. If $\cC$ is singular with singularity $P_0$, let $\cC(K)'
=\cC(K)\bsl\{P_0\}$. When $\cC$ is non-singular, write $\cC(K)' =\cC(K)$.

If $P,Q$ are points of $\cC(K)'$ then define $P*Q$ to be the third
intersection of the line $PQ$ with $\cC$. In particular, when $Q=P$, the line
$PQ$ is the tangent at $P$ and $P*P =P_t$ is the {\em tangential of $P$}. If
$P$ is an inflexion then $P_t =P$. 

Now choose any point $O$ of $\cC(K)'$ as the identity point for the group 
operation. Then   an operation $\oplus$ is defined on $\cC(K)'$
 as
follows:
\begin{equation}
   P\oplus Q= (P*Q)*O.
\end{equation}
The negative of $\oplus$ is written $-$.

\begin{figure}[ht]
\begin{center}
\begin{picture}(1,1)
\put(28,128){$O$}
\put(95,117){$P\oplus Q$}
\put(182,97){$P*Q$}
\put(88,72){$Q$}
\put(2,48){$P$}
\put(44,124.5){\circle*{3}}
\put(94,112.5){\circle*{3}}
\put(176,93){\circle*{3}}
\put(99,70){\circle*{3}}
\put(15.5,45.5){\circle*{3}}
\end{picture}
\scalebox{.5}{\includegraphics{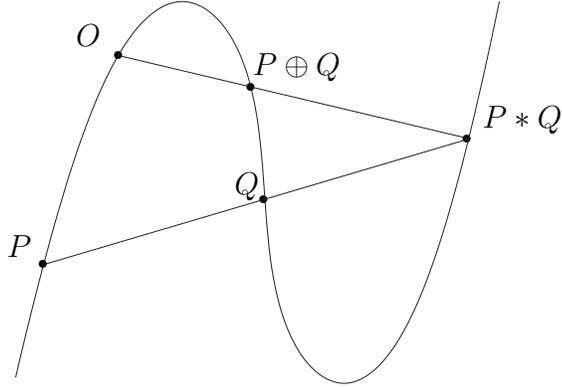}}
\end{center}
\caption{Abelian group law on an elliptic curve}
\label{fig112}
\end{figure}

Let the tangential $O_t$ at $O$ be denoted by  $N$.
\begin{theorem}
\label{group}
\begin{enumerate}
\item[\rm(i)]  The points of $\cC(K)'$ form a group $G$ with identity 
$O$ under the operation
$\oplus$.
\item[\rm(ii)] 
$-N=N$.
\item[\rm(iii)]\label{line1} 
Three points $P,Q,R$ of $\cC(K)'$ are collinear if and only if
$P \oplus Q \oplus R =N.$

\item[\rm(iv)] If $O$ is an inflexion$,$ then   three points 
$P,Q,R$ of $\cC(K)'$ are collinear if and only if
$P \oplus Q \oplus R =O.$
\end{enumerate}

\end{theorem}

If the characteristic of $K$ is not 2 or 3, then, as in Theorem \ref{flex1}, 
$\cC$ may be given by the form
$ F= Y^2Z - X^3 - c X Z^2 - d Z^3$. 
 
With the identity point $O=(0:1:0)$   the group law may be 
 expressed algebraically as follows.
With $P=(x_1:y_1:1)$ and $Q=(x_2:y_2:1)$,
$$
P\oplus Q= \left\{
\begin{array}{l}
    (0:1:0), \hspace{5.5cm} \mbox{if  $x_1 = x_2$ and  $y_1 \neq y_2$,}\\
    (\g^2-x_1-x_2:  -\g^3+2\g x_1 + \g x_2-y_1: 1),  \mbox{\quad otherwise},
\end{array}\right.
$$
where
$$\g= 
\begin{cases}
3x_1^2 + a/(2y_1), & \text{if } x_1=x_2,\\
(y_2-y_1)/(x_2-x_1), & \text{if } x_1 \neq x_2;
\end{cases} 
$$
see \cite[Section 6.6]{BF}.

Part (iii) of Theorem \ref{group} can be generalised to curves of
higher degree.

\begin{theorem} 
\label{thm2.4}
\begin{enumerate}
\item[\rm(i)] \label{2.4b} The six distinct points $P,Q,R,S,T, U$ of 
$\cC(K)'$ lie on a conic if and only if
$P\oplus Q \oplus R \oplus S \oplus T \oplus  U =2N$.

\item[\rm(ii)] A set of $3m$ points $P_1,P_2,\dots,P_{3m}$ of  $\cC(K)'$
lie on a curve of order $m$ if and only if $\sum_{i=1}^m P_i= mN$. 

\end{enumerate}

\end{theorem}

For cubics, geometric results have algebraic counterparts. Here is a sample
from \cite{L}, originally for the complex numbers, but applicable over any
field.
\begin{center}
\begin{tabular}{|l|l|}
\hline
&\\
&\\[-7mm]
\mcol{1}{|c|}{\bf Geometric formulation} &
\mcol{1}{|c|}{\bf Group-theoretic formulation}\\
&\\
&\\[-7mm]
\hline
&\\
&\\[-7mm]
$P$ and $Q$ have the same tangential  &     $2P=2Q$  or $2(P-Q)=0$\\
&\\
&\\[-7mm]
\hline
&\\
&\\[-7mm]
There exist four tangents from $P$ &    $2X \oplus P=N$ has four solutions\\ 
&\\
&\\[-7mm]
\hline
&\\
&\\[-7mm]
$P$ is a inflexion &    $3P=N$\\ 
&\\
&\\[-7mm]
\hline
&\\
&\\[-7mm]
$\cC$ has 9 inflexions & $3P=N$ has nine solutions\\
&\\
&\\[-7mm]
\hline
&\\
&\\[-7mm]
If $P$ and $Q$ are inflexions then           & If  $3P=N$ and  $3Q=N$, then\\
$R=P*Q$ is another inflexion;           & $P\oplus Q\oplus R=N$ implies $3R=N$\\
if $P\neq Q$ then $R \neq P,Q$          &     \\[3mm]
\hline
\end{tabular}
\end{center}

The calculations become more familiar, but not less complicated, if the point
$O$ is in fact an inflexion point. It is important to note that all different
choices for $O$ yield isomorphic groups.

\section{Classification of singular cubics}
\label{sec5}

An irreducible cubic $\cC$ over $K$ with a singular point $P_0$ has $2,1$ or
$0$ tangents lying over $K$ at $P_0$, which is correspondingly a {\em node},
{\em cusp} or {\em isolated double point}. 
Let $\cN_i^j$ indicate an 
irreducible singular cubic over $\Fq$ with $i$ rational inflexions and $j$
distinct rational tangents at the singularity; here, `rational' means 
`over $K$'.

When the characteristic of $K$ is $3$, there is one cubic $\cC=\cN_q^1$ of
particular note. It has the associated canonical form $F=ZY^2 - X^3$ and {\em
every} point of $\cN_q^1$ in $\P2(K)$ other than the singular point $P_0=(0:0:1)$
is an inflexion.

\begin{theorem}
\label{sing}
\begin{enumerate}
\item[\rm(i)] For an irreducible singular plane cubic curve $\cC$ over $\Fq$, 
with $\cC\neq\cN_q^1$,
\begin{enumerate}
\item[\rm(a)] there are $3$ collinear inflexions over $\ov{K};$
\item[\rm(b)] the inflexions are rational or lie
over a quadratic extension or a cubic extension.
\end{enumerate}

\item[\rm(ii)] For any $q$ there are
precisely four projectively distinct singular cubics. 
\end{enumerate}
\end{theorem}

In Table \ref{cubsing}, canonical forms are given in the cases of a node and a
cusp including $\cN_q^1$. For the canonical forms in the case of an isolated
double point, see \cite[Section 11.4]{James}.

\begin{table}[th]
\caption{Canonical forms for singular cubics}
\label{cubsing}
$$\begin{array}{cll}
\hline
&&\\
&&\\[-7mm]
\mbox{Symbol}  & \mcol{1}{c}{\quad q\equiv m\pmod{12}} & 
\mcol{1}{c}{\quad \mbox{Form}}\\
  & \mcol{1}{c}{\quad m} & \\
&&\\
&&\\[-7mm]
\hline
&&\\
&&\\[-7mm]
\cN_1^2  &\quad 3,9,2,8,5,11 &\quad XYZ - X^3 - Y^3 \\
\cN_3^2  &\quad 4,1,7 &\quad XY - X^3 - Y^3 \\
\cN_0^2  &\quad 4,1,7 &\quad XYZ - X^3 - \ga Y^3,
\mbox{ $\ga$ non-cube}\\
\cN_0^1  &\quad 3,9 &\quad ZY^2 - X^2Y - X^3 \\
\cN_q^1  &\quad 3,9 &\quad ZY^2  - X^3 \\
\cN_1^1  &\quad 2,8,4,1,5,7,11 &\quad ZY^2  - X^3 \\[3mm]
\hline
\end{array}
$$
\end{table}

There is a nice combinatorial/geometric characterisation of
singular cubics  due to Tallini Scafati \cite{TS}. See also 
\cite[Section 12.8]{James}.

\begin{theorem} 
Let $\cK$ be a set of $k$  points  in $\PG(2,q)$, with $q$ odd$,\ q > 11,$ 
and with no $4$ points of $\cK$ collinear. If
$\cK$ contains $4$ points $P, P_1,P_2, P_3$ such that
\begin{enumerate}
\item[\rm(i)]  there is no line through $P$ intersecting $\cK$ in $3$ points,
\item[\rm(ii)] any conic through $P$ and one of the $P_i$ meets $\cK$ in at 
most  $3$ other points,
\item[\rm(iii)] $k > q-\qa\sqrt{q} +\frac{19}{4},$ 
\end{enumerate}
then $\cK$ is contained in a rational cubic with a double point  at $P$.
\end{theorem}

Due to later results, see \cite[Sections 10.4, 10.5]{James} the lower bound in
(iii) in this theorem can be improved.

\section{\large Classification of non-singular cubics}
\label{sec6}

The following result from Section \ref{sec3} is recalled.
 \begin{theorem}
   If $\cC$ is a non-singular cubic curve defined over $K$ then
   $\cC$ has exactly $0, 1, 3$ or $9$ inflexion points in $\P2(K)$.
 \end{theorem}

\subsection{\large Non-singular cubics with nine rational inflexions}
\label{subsec6.1}

\begin{theorem}
\label{cub9}
A non-singular cubic $\cC$ with form $F$ and  nine rational inflexions exists
over $\Fq$ if and only if $q\equiv 1 \pmod{3}$, and then $F$ has
canonical form  
$$F = X^3 + Y^3 + Z^3  - 3c XYZ .$$
\end{theorem}
The nine inflexions are those given in (\ref{k9}).

\subsection{\large Non-singular cubics with three rational inflexions}
\label{subsec6.2}

\begin{theorem}
\label{cub3}
A non-singular cubic $\cC$ with form $F$ and  three rational inflexions exists
over $\Fq$ for all $q$. The inflexions are necessarily collinear.
\begin{enumerate}
\item[\rm(i)]\label{3i} If the inflexional tangents are 
concurrent$,$  the canonical forms are as follows$:$
\begin{enumerate}
\item[\rm (a)] $q\equiv 0,2\pmod{3},$
\begin{eqnarray*}
F & = & XY(X + Y) + Z^3;
\end{eqnarray*}
\item[\rm (b)] $q\equiv 1\pmod{3},$
\begin{eqnarray*}
F & = & XY(X + Y) + Z^3,\\
F & = & XY(X + Y) + \ga Z^3,\\
F & = & XY(X + Y) + \ga^2 Z^3,
\end{eqnarray*}
where $\ga$ is a primitive element of $\Fq$.
\end{enumerate}

\item[\rm(ii)]\label{3ii}  If the inflexional tangents are non-concurrent,  
the canonical form is as follows$:$
$$F= XYZ + e(X + Y + Z)^3,$$
$e\neq 0, -1/27$.
\end{enumerate}

\end{theorem}

In case (i), the inflexions are
\[
(1:0:0),\ (0:1:0),\ (1:-1;0);
\]
in case (ii), the inflexions are
\[
(0:1:-1),\ (1:0:-1),\ (1:-1:0).
\]

\subsection{\large Non-singular cubics with one rational inflexion}
\label{subsec6.3}

For $q=2^h$, the {\em trace} of an element $x\in\Fq$ is
\[
\gt(x) = x+ x^2 + x^4 + \cdots + x^{2^{h-1}}.
\]
\begin{theorem}
\label{flex1}
A non-singular$,$ plane$,$ cubic curve defined over $\Fq,\ q=p^h,$ with at
least one inflexion has one of the following canonical forms $F$.
\begin{enumerate}
\item[\rm (i)] $p\neq 2,3,$
$$ F= Y^2Z - X^3 - cXZ^2 - dZ^3,$$
where $4c^3 + 27d^2 \neq 0.$ 

\item[\rm (ii)] $p = 3,$
\begin{enumerate}
\item[\rm (a)]
$$ F= Y^2Z - X^3 - bX^2Z - dZ^3,$$
where $bd \neq 0;$

\item[\rm (b)]
$$ F' = Y^2Z - X^3 - cXZ^2 - dZ^3,$$
where $c \neq 0.$  
\end{enumerate}

\item[\rm (iii)] $p = 2,$
\begin{enumerate}
\item[\rm (a)]
$$ F= Y^2Z + XYZ + X^3 + bX^2Z + dZ^3,$$
where $b= 0$ or a fixed element of trace $1,$ and $c\neq 0;$

\item[\rm (b)]
$$ F= Y^2Z +YZ^2 + eX^3 + cXZ^2 + dZ^3,$$
where $e=1$ when $q\equiv 0,2\pmod{3}$ and $e=1,\ga,\ga^2$ when 
$q\equiv 1\pmod{3},$ with
$\ga$ a primitive element of $\Fq;$ also$,\ d=0$ or a given element
of trace $1$.
\end{enumerate}
\end{enumerate} 
\end{theorem}

 A complete discussion and classification of cubic curves over finite fields
 may be found in \cite{James}.

\subsection{\large Non-singular cubics with no rational inflexions}
\label{subsec6.4}

\begin{theorem}
A non-singular$,$ plane$,$ cubic curve defined over $\Fq,\ q=p^h,$ with 
no rational inflexion has one of the following canonical forms $F$.
\begin{enumerate}
\item[\rm (i)]
 $q\equiv 2 \pmod{3}$, $$ F = Z^3 - 3c(X^2 - d X Y +
Y^2) Z - (X^3 - 3 X Y^2 + d Y^3),$$ where $T^3 - 3 T + d$ is
irreducible.

\item[\rm (ii)]
$q\equiv 1\pmod{3},$

\begin{enumerate}
\item[\rm (a)]
  $$F= X^3 + \ga Y^3 + \ga^2 Z^3 - 3c XYZ,$$
with $\ga$ a primitive element of $\Fq$.

\item[\rm (b)]
  $$F= X Y^2 + X^2Z + e Y Z^2 - 
c(X^3 + e Y^3 + e^2 Z^3 - 3e XYZ),$$
with $\ga$ a primitive element of $=\Fq$ and $e=\ga,\ga^2$.
\end{enumerate} 

\item[\rm (ii)]
$q\equiv 0\pmod{3},$
$$F = X^3 + Y^3 + cZ^3  + dX^2Z + d X Y^2 + d^2 X^2 + 
d YZ^2,$$
where $c\neq 1$ and $T^3 +d T - 1$ is a fixed irreducible polynomial.
\end{enumerate} 
\end{theorem}

\section{Number of rational points on a cubic}
\label{sec7}

With $ N_1$ the number of rational points on a curve $\cF$, consider the
case that $\cF$ is a non-singular plane cubic $\cC$.
The Hasse bound states that
\begin{equation}
\label{hbound}
  q+1 - 2 \sqrt{q} \leq N_1 \leq q+1 + 2\sqrt{q}. 
\end{equation}

The next result shows what values in the range actually occur.
 For any integer $M$ and any prime divisor $\ell$, let $v_\ell(M)$ be
the highest power of $\ell$ dividing $M$; that is, $\prod_{\ell}
\ell^{v_\ell(M)}$ is the prime decomposition of $M$. 

\begin{theorem}
There exists a non-singular plane cubic over $\Fq,\ q=p^h,$ with precisely
$N_1=q+1 -t$ rational points$,$ where $|\,t\,|\leq 2\sqrt{q},$ in the
cases listed in Table \ref{cubno}. Below$,$ $G_\cC$ is the corresponding 
group formed by the points of the  cubic.
\begin{table}[ht]
\caption{Values of $t$}
\label{cubno}
$$\begin{array}{lccc}
\hline
&&&\\
&&&\\[-8mm]
 & \quad t &\quad p &\quad h \\
&&&\\
&&&\\[-8mm]
\hline
&&&\\
&&&\\[-8mm]
(1) & \quad t \not\equiv 0\ ({\rm mod}\ p) &  &   \\[1mm]
(2) & \quad t=0 &\quad  &\quad odd \\[1mm]
(3) & \quad t=0 &\quad p\not\equiv 1\pmod{4} &\quad even\\[1mm]
(4) & \quad t=\pm\sqrt{q} &\quad p\not\equiv 1\pmod{3} &\quad even
\\[1mm]
(5) & \quad t=\pm 2\sqrt{q} & &\quad even\\[1mm]
(6) & \quad t=\pm \sqrt{2q} & p=2 &\quad odd \\[1mm]
(7) & \quad t=\pm \sqrt{3q} & p=3 &\quad odd \\[3mm]
\hline
\end{array}
$$
\end{table}

\begin{enumerate}
\item[\rm(1)] $G_\cC =
\bZ/(p^{v_{p}(N_1)})\times\prod_{\ell\neq p}\left(
\bZ/(\ell^{r_{\ell}})\times \bZ/(\ell^{s_{\ell}})\right),$\\
$ \mbox{with $r_\ell +
s_\ell= v_\ell(N_1)$ and $\min(r_\ell,s_\ell) \leq v_\ell(q-1);$}$
\item
[\quad$\begin{array}{c}(2)\\(3)\end{array}$]   
$ G_\cC=\left\{\begin{array}{ll}
       \bZ/(q + 1)  &
\mbox{for $q\not\equiv -1\pmod{4},$}\\
      \mbox{$ \bZ/(q + 1)$ or $ \bZ/(2)\times \bZ/((q+1)/2)$} &
\mbox{for $q\equiv -1\pmod{4};$}
\end{array}\right.$
\item[\rm(4)]  $G_\cC=\bZ/(N_1);$
\item[\rm(5)] $G_\cC =  \bZ/(\sqrt{N_1})\times  
\bZ/(\sqrt{N_1}),\ N_1=(\sqrt q\pm 1)^2;$
\item[\rm(6)]  $G_\cC=\bZ/(N_1);$
\item[\rm(7)]  $G_\cC=  \bZ/(N_1).$
\end{enumerate}
\end{theorem}

The range of $t$ is due to Waterhouse \cite{W} and the corresponding groups 
independently to R{\"u}ck \cite{R} and Voloch \cite{V}.

Let $N_q(1)$ denote the maximum number of rational points on any
non-singular cubic over $\Fq$ and $L_q(1)$ the minimum number. The prime power
$q=p^h$ is {\em exceptional} if $h$ is odd$,$ $h \geq 3,$ and $p$ divides
$\lfloor2\sqrt{q}\rfloor$.
\begin{cor} The bounds $N_q(1)$ and $L_q(1)$ are as follows$:$
\begin{enumerate}
\item[\rm (i)]
$ N_q(1) = \left\{ \begin{array}{ll}
q + \lfloor2\sqrt{q}\rfloor, & \mbox{if $q$ is exceptional}\\ 
q + 1 + \lfloor2\sqrt{q}\rfloor, & \mbox{if $q$ is non-exceptional};
                  \end{array}\right. $
\item[\rm (ii)]
$ L_q(1) = \left\{ \begin{array}{ll}
q + 2 - \lfloor2\sqrt{q}\rfloor, & \mbox{if $q$ is exceptional}\\ 
q + 1 - \lfloor2\sqrt{q}\rfloor, & \mbox{if $q$ is non-exceptional}.
                  \end{array}\right.$
\end{enumerate}
\end{cor}

\begin{cor}
The number $N_1$ takes every value between $q + 1 - \lfloor2\sqrt{q}\rfloor$ and
$q + 1 + \lfloor2\sqrt{q}\rfloor$ if and only if {\rm (a)} $q =p$ or 
{\rm (b)} $q=p^2$ with $p=2$ or $p=3$ or $p\equiv 11\pmod{12}.$
\end{cor}

\begin{remark}
The only exceptional $q < 1000$ is $q=128$.
\end{remark}

\begin{theorem}
The number of points $N_1$ on a non-singular cubic$,$ for which  the number
$n$ of rational inflexions is $n=0,1,3,9,$ satisfies the following$:$
\begin{enumerate}
\item[\rm (i)] If $n=0,$ then $N_1 \equiv 0\pmod{3};$
\item[\rm (ii)] If $n=1,$ then $N_1 \equiv \pm 1\pmod{3};$
\item[\rm (iii)] If $n=3,$ then $N_1 \equiv 0\pmod{3};$
\item[\rm (iv)] If $n=9,$ then $N_1 \equiv 0\pmod{9}.$
\end{enumerate}
\end{theorem}

Let $A_q$ be the total number of isomorphism classes and $P_q$ the total
number of projective equivalence classes. Also,
$n_i$ for $i=0,1,3,9$  is the number of projective equivalence classes
with exactly $i$ rational inflexions.  Hence
$$A_q  =  n_9 + n_3 + n_1,\qquad
P_q  =  n_9 + n_3 + n_1 + n_0.
$$
\begin{theorem}
\begin{enumerate}
\item[\rm(i)] ${\dis A_q = 2q + 3 + \left(\frac{-4}{q} \right) +
                     2\left(\frac{-3}{q} \right)};$
\item[\rm(ii)] ${\dis P_q = 3q + 2 + \left(\frac{-4}{q} \right) +
                     \left(\frac{-3}{q} \right)^2 
                     + 3\left(\frac{-3}{q} \right)}.$
\end{enumerate}
Here$,$ the following Legendre--Jacobi symbols
are used$:$
\begin{eqnarray*}
\left(\frac{-4}{c} \right) & = & \left\{\begin{array}{rl}
1 & \mbox{if $c\equiv 1 \pmod 4$},\\
0 & \mbox{if $c\equiv 0 \pmod 2$},\\
-1 & \mbox{if $c\equiv -1 \pmod 4$};      \end{array}\right. \\
\left(\frac{-3}{c} \right) & = & \left\{\begin{array}{rl}
1 & \mbox{if $c\equiv 1 \pmod 3$},\\
0 & \mbox{if $c\equiv 0 \pmod 3$},\\
-1 & \mbox{if $c\equiv -1 \pmod 3$}.      \end{array}\right.
\end{eqnarray*}
\end{theorem}

\begin{table}[th]
\caption{Number of inequivalent cubics}
\label{ineqcub}
$$\begin{array}{ccccccc}
\hline
&&\\
&&\\[-7mm]
 q\equiv m\!\!\!\pmod{12} & n_9 & n_3 & n_1 & n_0 & A_q & P_q
\\
 \mcol{1}{c}{\quad m} & &&&&& \\
&&&&&\\
&&&&&\\[-8mm]
\hline
&&&&&\\
&&&&&\\[-7mm]
3  & 0 & q-1 & q +3 & q-1 & 2q +2& 3q+1\\[1mm]
9  & 0  & q-1 & q +5 & q-1 & 2q +4 & 3q+3 \\[1mm]
2,8  & 0  & q-1 & q +2 & q-1 &2q +1& 3q \\[1mm]
4  & \twe(q+8) & \thi(2q+4) & \qa(5q + 12) & q+1 & 2q +5& 3q+6   \\[1mm]
1  & \twe(q+11) & \thi(2q+4) & \qa(5q + 15) & q+1&2q +6 & 3q+7 \\[1mm]
7  & \twe(q+5) &\thi(2q+4)  & \qa(5q + 9) & q+1 & 2q +4 & 3q+5\\[1mm]
5  & 0 &  q-1& q +3 & q-1 & 2q +2 & 3q+1 \\ [1mm]
11 & 0 &  q-1& q +1 & q-1 &2q & 3q-1\\[3mm]
\hline
\end{array}
$$
\end{table}

The number of inequivalent types of cubic with a fixed number of rational
points can also be given. Let $A_q(t)$ and $P_q(t)$ be the number of
inequivalent non-singular cubics with exactly $q + 1 -t$ rational points under
isomorphism and projective equivalence. So $$ A_q = \sum_t A_q(t), \qquad P_q
= \sum_t P_q(t).$$ The values of $A_q(t)$ and $P_q(t)$, due to Schoof
\cite{Sc}, are also given in
\cite[Section 11.11]{James}.

\section{Some new applications  in finite geometries} 
\label{appl}

Much of finite geometries is concerned with maximal sets of points in
$\PG(2,q)$ obeying various geometrical conditions: such sets are often of
considerable interest also in algebraic coding theory. For example, a key
result in the theory of MDS codes has as its foundation a famous theorem of
the late B.~Segre. This result asserts that, for $q$ odd, a set of points with
no 3 collinear has size at most $q+1$ with equality if and only
if the set is the point set of a non-degenerate conic.

The next result, found independently by A. Zirilli and P.M. Neumann, 
see  \cite{B}, 
is usually phrased using elliptic curves, that is, non-singular cubic curves
with an inflexion point. However, as is seen below, this assumption is not
necessary.

\begin{theorem}  \label{curvecaps}
If a non-singular cubic curve $\cC$ has an even number $k$ of rational
points$,$ then there exists a set $\cS$ of $k/2$ points of $\cC$ with no three
collinear.
\end{theorem}

\begin{proof}
Let $G$ be the abelian group obtained from $\cC$, using the general
construction of Section \ref{grouplaw} and the notation there. Then, from
Theorem~\ref{group}, three points on $\cC$ are collinear if
and only if they add up to $N$. Let $H$ be the subgroup of index 2 in $G$.
There is another coset $K$ of $H$ in $G$ so that $G$ is the disjoint union of
$H$ and $K$. There are two cases:
\begin{enumerate}[(i)]
\item $N$ lies in $H$;
\item $N$ lies in $K$.
\end{enumerate}

In case (i), take $\cS$ to be the set $K$. Suppose that $P,Q,R$ are in
$K$. Then $P\oplus Q$ must be in $H$ so that $P\oplus Q\oplus R$ must be in
$K$. In particular, $P\oplus Q\oplus R$ cannot be equal to $N$, which is in
$H$. Therefore $\cS$ is a set of $k/2$ points with no 3 collinear.

In case (ii), take $\cS$ to be the set $H$. Let $P,Q,R$ be any 3 points
of $\cS$. Since $H$ is a subgroup, $P\oplus Q\oplus R$ is in $H$. In
particular, $P\oplus Q\oplus R$ cannot equal $N$ since $N$ is in $K$. Thus
$\cS$ is a set of $k/2$ points, with no 3 collinear.
\end{proof}

\begin{remark}  
It is  possible  that such sets $\cS$  are not maximal when
considered as sets of points in $\PG(2,q)$ with  no 3 collinear.
\end{remark}

So far, sets of points with no three collinear have been considered. As any two
points define a unique line, the next step is to try to find a result
analogous to Theorem \ref{curvecaps} for higher degree curves.

Theorem~\ref{curvecaps} can be generalised as follows.

\begin{theorem}\label{general} 
Let $\cC$ be a non-singular cubic curve in $\PG(2,q)$ containing exactly $n$
points and with a cyclic group structure. Suppose the integer $r$ divides $n$.
Then there exists a set $\cS$ of $n/r$ points on $\cC$ satisfying the
following condition$:$ no $3k$ points of $\cS$ lie on any curve of degree $k$
other than $\cC$ whenever $1 \leq k < \lceil r/3 \rceil$.
\end{theorem}

\begin{proof}
The group $G$ has a subgroup $H$ of order $n/r$ and index $r$.
The cosets of $H$ are denoted by $H=H_0$, $H_1, H_2,\dots H_{r-1}$, where
$H_i \oplus  H_j = H_{i + j \pmod r}$.  

  Let $H_j$ be the coset containing the point $N$. Take $\cS = H_i$ where $i$
  is to be determined. Choose any $3k$ points $P_1,P_2,\dots, P_{3k}$ in
  $\cS$. Their sum lies in the coset $3k H_1 = H_{3k \pmod r}$. Using Theorem
  \ref{thm2.4}, $\cS$ has the desired property if the sum of these $3k$ points
  is always different from $kN$. This will follow from showing that the cosets
  $H_{3ki \pmod r}$ and  $H_{kj \pmod r}$ are unequal; that is 
  $k(3i-j) \not\equiv 0 \pmod r$.

 There are three cases to consider:
 \begin{enumerate}[(i)]
 \item 3 does not divide $r$;
 \item 3 divides $r$ but 3 does not divide $j$;
 \item 3 divides both $r$ and $j$.
 \end{enumerate}

 In case (i), 3 does not divide $r$. Then 3 has a multiplicative inverse $u$
 modulo $r$. Put $i =uj + u$. Then $3i-j \equiv 1 \pmod r$. Thus, if $k(3i-j)
 \equiv 0 \pmod r$, then $k \equiv 0 \pmod r$. The hypotheses imply that $0 < k
 < r$. So $\cS=H_i$ has the desired property.
 
 In case (ii), if $j \equiv 2 \pmod 3$ then put $i = (j+1)/3$. If $j \equiv 1
 \pmod 3$ then put $i=(j-1)/3$. Thus $3i-j$ is either 1 or $-1$ modulo $r$.
 Hence, if $k(3i-j) \equiv 0 \pmod r$, then $k \equiv 0 \pmod r$. As in case
 (i), this implies that $\cS$ has the required property.
 
 In case (iii), put $i=j/3+1$. Then $3i-j=3$. If $k(3i-j) \equiv 0 \pmod
 r$ then $3k \equiv 0 \pmod r$. But this contradicts the assumed bounds on
 $k$. Again this implies that $\cS$ has the desired property. 
  \end{proof}

\begin{remark}
    In this proof, only the fact that $G/H$ is cyclic is used. Strictly
    speaking, the assumption that $G$ is cyclic can weakened to merely
    assuming that $G/H$ is cyclic.
\end{remark}

\begin{remark}
  In the statement of the theorem, the restriction that $k < \lceil r/3
  \rceil$ is only required in case (iii). In the other two cases, it is
  sufficient to assume that $k < r$.
\end{remark}

\begin{remark}
Concerning the sets constructed in Theorem~\ref{curvecaps}, Voloch
\cite{V2} has shown that, in many cases, they cannot be extended to larger
sets with no 3 points collinear. In \cite{AB}, it is shown that these results
of Voloch can be strengthened and generalised. 
\end{remark}

\thanks{The research of the first author is supported by grants from NSERC.
The research of the third author is supported by grants from 
ARP and NSERC.}



\begin{tabular}{lll}
A.A. Bruen & J.W.P. Hirschfeld & D.L. Wehlau \\
Department of Electrical  & Department of & Department of Mathematics  \\
and Computer  Engineering & Mathematics & and Computer Science \\
University of Calgary & University of Sussex & Royal Military College \\
Calgary & Brighton & Kingston \\ 
Alberta T2N 1N4 & East Sussex BN1 9RF & Ontario  K7K 7B4  \\ 
Canada & United Kingdom & Canada\\
{\tt bruen@ucalgary.ca} & {\tt jwph@sussex.ac.uk} & {\tt wehlau@rmc.ca}
\end{tabular}
\end{document}